\documentclass[12 pt]{article}
\usepackage{amssymb}
\usepackage{amsmath}
\usepackage{amsthm}
\allowdisplaybreaks
\usepackage{amscd}
\usepackage{graphicx}
\usepackage[all,cmtip]{xy}

\usepackage{geometry}

\usepackage{ifpdf}
\ifpdf
 \usepackage[colorlinks=true,linkcolor=blue,final,backref=page,hyperindex]{hyperref}
\else
  \usepackage[colorlinks,final,backref=page,hyperindex,hypertex]{hyperref}
\fi
\def\vbar{\mathchoice{\vrule height6.3ptdepth-.5ptwidth.8pt\kern- .8pt}
{\vrule height6.3ptdepth-.5ptwidth.8pt\kern-.8pt} {\vrule
height4.1ptdepth-.35ptwidth.6pt\kern-.6pt} {\vrule
height3.1ptdepth-.25ptwidth.5pt\kern-.5pt}}

\def\<{\langle}
\def\>{\rangle}
\def\a{\alpha}
\def\b{\beta}

\newtheorem{thm}{Theorem}[section]

\newtheorem{cor}[thm]{Corollary}
\newtheorem{pro}[thm]{Proposition}
\newtheorem{ex}[thm]{Example}

\theoremstyle{definition}
\newtheorem{defi}{Definition}[section]

\theoremstyle{remark}
\newtheorem{rmk}{Remark}[section]
\begin{document}
\title{On deformation cohomology of compatible Hom-associative algebras}
\author{
{ Taoufik  Chtioui$^{1}$%\footnote { Corresponding author,  E-mail: chtioui.taoufik@yahoo.fr}
 \ and
Ripan Saha$^{2}$\footnote {Corresponding author,  E-mail:  ripanjumaths@gmail.com}
}\\
{\small 1.  University of Sfax, Faculty of Sciences Sfax,  BP
1171, 3038 Sfax, Tunisia} \\
{\small 2.  Department of Mathematics, Raiganj University
Raiganj 733134, West Bengal, India }}
\date{}
\maketitle

\begin{abstract}
In this paper, we consider compatible Hom-associative  algebras as a twisted version of compatible associative algebras.
Compatible Hom-associative  algebras are characterized as Maurer-Cartan elements in a suitable bidifferential
graded Lie algebra. We also define a cohomology theory for compatible Hom-associative  algebras generalizing  the classical case. As applications of cohomology, we study abelian extensions and
deformations of compatible Hom-associative  algebras.
\end{abstract}



%\tableofcontents

\numberwithin{equation}{section}

\section*{Introduction}
Two (linear) product operations of a certain kind defined on the same vector space are said to be compatible if their linear combinations are still of the same kind. In recent years a lot of studies have been done on various types of compatible algebraic structures, for example, compatible associative algebras \cite{Odesskii2}, compatible Lie algebras \cite{golu1, golu2, golu3}, compatible Lie bialgebras \cite{wu bai, mar}, etc. Compatible associative algebras have appeared in connection with Cartan matrices of affine Dynkin diagrams, infinitesimal bialgebras, integrable matrix equations, and quiver representations. Bolsinov and Borisov \cite{bol} showed a close connection between compatible Lie algebras and compatible Poisson structures via dualization. Compatible algebras showed up in the study of loop algebras over Lie algebras \cite{golu2}, elliptical theta functions\cite{Odesskii1}, classical Yang-Baxter equation\cite{Odesskii3}, and principal chiral fields\cite{golu1}. Compatible algebraic structures also appeared in many other interesting works of mathematical physics, see \cite{uchino, mag-mor, kos, wu, stro}. In \cite{cht das mab}, the authors studied the cohomology and deformation of compatible associative algebras. In \cite{comp-lie}, the authors characterized compatible Lie algebras as Maurer-Cartan elements and studied cohomology and deformation of compatible Lie algebras. In \cite{Das comp HomLie}, the authors studied compatible Hom-Lie algebras generalizing the work of \cite{comp-lie}.

\medskip

In this paper, we consider the twisted analog of compatible associative algebras, that is, compatible Hom-associative algebras. We study the cohomology of compatible Hom-associative algebras, generalizing the cohomology of compatible associative algebras. We also show some  relationship between the cohomology of compatible Hom-associative algebras and cohomology of compatible Hom-Lie algebras defined in \cite{Das comp HomLie}. We define and study a one-parameter formal deformation theory of compatible Hom-associative algebras. We show that the infinitesimal is a $2$-cocycle and our cohomology is the `deformation cohomology', that is, our cohomology controls the formal deformation. We also study the abelian extension of the compatible Hom-associative algebras and show that equivalence classes of the extensions are one-to-one correspondence with the elements of the second cohomology group.

\medskip

The paper is organized as follows. In Section \ref{section:background}, we recall the basics, cohomology of Hom-associative algebras, and the Gerstenhaber bracket. In Section \ref{section-ca}, we introduce compatible Hom-associative algebras and compatible bimodules. We also construct the graded Lie algebra whose Maurer-Cartan elements are precisely compatible Hom-associative structures. In Section \ref{sec:cohomology}, we introduce the cohomology of a compatible Hom-associative algebra with coefficients in a compatible bimodule. In Section \ref{section:deformation}, we study one-parameter formal deformations of compatible Hom-associative algebras from cohomological points of view. Finally,  in Section \ref{section: abelian ext} we study abelian extensions of compatible Hom-associative algebras and their relationship with the second cohomology group.

\subsection*{Notations}
Given a Hom-associative algebra $(A, \mu,\a)$, we use the notation $a \cdot b$ for the element $\mu (a, b),$ for $a,b \in A$. If $(M, \ell, r,\b)$ is a bimodule over the associative algebra $A$, we use the same notation dot for left and right $A$-actions on $M$, i.e., we write $a \lhd m$ for the element $\ell(a,m)$ and $m \rhd a$ for the element $r (m,a)$.

All vector spaces, (multi)linear maps and tensor products are over a field $\mathbb{K}$ of characteristic $0$. All vector spaces are finite-dimensional.

%%%%%%%%%%%%%%%%%%%%%%%%%%%%%%%%%%%%%%%%%%%%%%%%%%%%%%%%%%%%%%%
\section{Preliminaries and basics  }
\label{section:background}
%%%%%%%%%%%%%%%%%%%%%%%%%%%%%%%%%%%%%%%%%%%%%%%%%%%%%%%%%%%%%%%
In this section, we recall the (Hochschild) cohomology theory of Hom-associative algebras , the Gerstenhaber bracket and bidifferential graded Lie algebras.  Our main references are \cite{hoch, gers2,gers,loday-book, das proc}.

\subsection{Cohomology of Hom-associative algebras}

\begin{defi}
A Hom-associative algebra is a triple $(A,\mu,\a)$  in which $A$   is a vector space endowed with a bilinear map  $\mu : A \otimes A \rightarrow A, (a, b) \mapsto a \cdot b$  and a linear map $\a: A \to A$, satisfying the following condition
\begin{align*}
& \a(a \cdot b)=\a(a) \cdot \a(b), \\
& ( a \cdot b) \cdot \a(c) = \a(a) \cdot ( b \cdot c), ~\text{ for } a, b, c \in A.
\end{align*}
\end{defi}
We may also denote a Hom-associative algebra simply by $(A,\a)$ when the multiplication map is clear from the context.

\begin{defi}
A bimodule  (or a  representation) over a Hom-associative algebra $A$ consists of a vector space $M$ together with bilinear maps (called left and right $A$-actions) $\ell : A \otimes M \rightarrow M, (a, m) \mapsto a \lhd m$ and $r : M \otimes A \rightarrow M, (m, a) \mapsto m  \rhd a$ and a linear map $\b: M \to M$ satisfying the following compatibilities
\begin{align*}
& \b(a \lhd m)=\a(a) \lhd \b(m), \quad \b(m \rhd a)=\b(m) \rhd \a(a),\\
& (a \cdot b ) \lhd  \b(m) = \a(a)  ( b \lhd  m),  \\
& ( a \lhd m ) \rhd \a( b) = \a(a)  \lhd(m \rhd  b) ,\\
& (m \rhd a) \rhd \a(b) = \b(m) \rhd ( a \cdot b),
\end{align*}
for $a, b \in A$ and $m \in M$.  Such a bimodule  is denoted by $(M,\ell,r,\b)$.
\end{defi}

It follows that $A$ is an $A$-bimodule with both left and right $A$-actions are given by the algebra multiplication map.

Let $(M,\ell,r,\beta)$ be a representation of a Hom-associative  algebra $(A,\mu,\alpha)$.
In the sequel, we always assume that $\beta$ is invertible. For all $a\in A,u\in M,\xi\in M^*$,
define $r^*:A\longrightarrow gl(M^*)$ and $\ell^*:A\longrightarrow gl(M^*)$ as
usual by
$$\langle r^*(a)(\xi),u\rangle=-\langle\xi,r(a)(u)\rangle,\quad \langle \ell^*(a)(\xi),u\rangle=-\langle\xi,\ell(a)(u)\rangle.$$
Then define $r^\star:A\longrightarrow gl(M^*)$ and $\ell^\star:A\longrightarrow gl(M^*)$ by
\begin{eqnarray}
  \label{eq:1.3}r^\star(a)(\xi):=r^*(\alpha(a))\big{(}(\beta^{-2})^*(\xi)\big{)},\\
   \label{eq:1.4}\ell^\star(a)(\xi):=\ell^*(\alpha(a))\big{(}(\beta^{-2})^*(\xi)\big{)}.
\end{eqnarray}

\begin{thm}\label{dual-rep}
Let $(M\ell,r,,\beta)$ be a representation of a Hom-associative algebra $(A,\mu,\alpha)$. Then $(M^*,r^\star-\ell^\star,-\ell^\star,(\beta^{-1})^*)$ is a representation of $(A,\mu,\alpha)$, which is called the dual representation of $(M,\ell,r,,\beta)$.
\end{thm}

In \cite{amm-ej-makh, das proc, mak-sil} the authors define a Hochschild-type cohomology of a Hom-associative algebra and study the one-parameter formal deformation theory for these type of algebras. Let $(A, \mu,\a)$ be a Hom-associative algebra and $(M,\ell,r,\b)$ be a bimodule. For each $n \geq 1$,  the group of $n$-cochains is defined as
\begin{align*}
C^n_{\alpha,\b}(A, M):= \{ f : A^{\otimes n} \rightarrow M |~ \b\circ f (a_1, \ldots, a_n ) = f (\alpha (a_1), \ldots, \alpha (a_n)) \}. 
\end{align*}\newpage
The differential $\delta_{\alpha,\b} : C^n_{\alpha, \b} (A, M) \rightarrow C^{n+1}_{\alpha, \b} (A, M)$ is given by
\begin{align}\label{hoch-diff}
(\delta_{\a,\beta} f)(a_1, \ldots, a_{n+1}) =~& \alpha^{n-1} (a_1) \lhd f (a_2, \ldots, a_{n+1}) \\
~&+ \sum_{i=1}^n (-1)^{i} f (   \alpha (a_1), \ldots, \alpha (a_{i-1}), a_i\cdot  a_{i+1}, \alpha (a_{i+2}), \ldots, \alpha (a_{n+1}) ) \nonumber \\
~&+ (-1)^{n+1} f (a_1, \ldots, a_n) \rhd \alpha^{n-1} (a_{n+1}). \nonumber
\end{align}
for $a_1, \ldots, a_{n+1} \in A$.
The cohomology of this complex is called the Hochschild cohomology of $A$ with coefficients in the bimodule $(M,\ell,r,\b)$ and the cohomology groups are denoted by
$H^\bullet_{\a,\b}(A,M)$.

In \cite{amm-ej-makh, mak-sil} , authors  construct a graded Lie algebra structure on  $C^{\ast + 1}_\a (A, A) = C^\ast _\a(A, A) [1] = \oplus_{n \geq 0} C^{n+1}_\a (A, A)$ and defined a graded Lie bracket (called the Gerstenhaber bracket) on $C^{\ast +1 }_\a(A,A)$ by
\begin{align}
[f, g]_\a  =~& f \diamond g - (-1)^{mn}~ g \diamond f, ~~ \text{ where }  \label{gers-brk}\\
&(f \diamond g)(a_1, \ldots, a_{m+n+1} )\nonumber \\
=~&  \sum_{i = 1}^{m+1} (-1)^{(i-1)n}~f ( \a^{n}(a_1), \ldots, \a^{n}(a_{i-1}), g ( a_i, \ldots, a_{i+n}), \ldots, \a^{n}(a_{m+n+1})), \nonumber
\end{align}
for $f \in C^{m+1} _\a(A,A)$ and $g \in C^{n+1}_\a(A,A)$. The importance of this graded Lie bracket is given by the following characterization of associative structures.

\begin{pro}
Let $A$ be a vector space and $\mu: A^{\otimes 2} \rightarrow A$ be a bilinear map on $A$. Then $\mu$ defines a Hom-associative structure on $A$ if and only if $[\mu,\mu]_\a = 0$.
\end{pro}

Let $(A,\mu ,\a)$ be a Hom-associative algebra. Then it follows from (\ref{hoch-diff}) and (\ref{gers-brk}) that the Hochschild coboundary map $\delta_\a : C^n_\a (A,A) \rightarrow C^{n+1} _\a(A,A)$ for the cohomology of $A$ with coefficients in itself is given by
\begin{align}\label{hoch-diff-brk}
\delta_\a f = (-1)^{n-1} [\mu, f]_\a, ~\text{ for } f \in C^n_\a(A,A).
\end{align}

%------------------------------------------------------------------------------------
\subsection{Bidifferential graded Lie algebras}
%--------------------------------------------------------------------------------------
Now,  we recall bidifferential graded Lie algebras \cite{comp-lie}. Before that,
let us first give the definition of a differential graded Lie algebra.

\begin{defi}
A differential graded Lie algebra is a triple $(L= \oplus L^i, [~,~], d)$ consisting of a graded Lie algebra together with a degree $+1$ differential $d : L \rightarrow L$ which is a derivation for the bracket $[~,~]$.
\end{defi}
An element $\theta \in L^1$ is said to be a Maurer-Cartan element in the differential graded Lie algebra $(L, [~,~], d)$ if $\theta$ satisfies
\begin{align*}
d \theta + \frac{1}{2} [\theta, \theta] = 0.
\end{align*}

\begin{defi}
A bidifferential graded Lie algebra is a quadruple $(L = \oplus L^i, [~, ~], d_1, d_2)$ in which the triples $(L, [~,~], d_1)$ and $(L, [~,~], d_2)$ are differential graded Lie algebras additionally satisfying
%\begin{align*}
$d_1 \circ d_2 + d_2 \circ d_1 = 0.$
%\end{align*}
\end{defi}

\begin{rmk}\label{dgla-bdgla}
Any graded Lie algebra can be considered as a bidifferential graded Lie algebra with both the differentials $d_1$ and $d_2$ to be trivial.
\end{rmk}

%\begin{exam}\label{bidiff-exam}
%Let $\mathfrak{g}$ be a vector space and $\alpha : \mathfrak{g} \rightarrow \mathfrak{g}$ be a linear map. Consider the Nijenhuis-Richardson graded Lie algebra $(C^{\ast +1}_\mathrm{Hom} (\mathfrak{g}, \mathfrak{g} ), [~,~]_\mathsf{NR})$ defined in the previous subsection. Hence it follows from the previous remark that $(C^{\ast +1}_\mathrm{Hom} (\mathfrak{g}, \mathfrak{g} ), [~,~]_\mathsf{NR}, d_1 = 0, d_2 = 0)$ is a bidifferential graded Lie algebra.
%\end{exam}

\begin{defi}
Let $(L, [~,~], d_1, d_2)$ be a bidifferential graded Lie algebra. A pair of elements $(\theta_1, \theta_2) \in L^1 \oplus L^1$ is said to be a Maurer-Cartan element if
\begin{itemize}
\item[(i)] $\theta_1$ is a Maurer-Cartan element in the differential graded Lie algebra $(L, [~,~], d_1)$;
\item[(ii)] $\theta_2$ is a Maurer-Cartan element in the differential graded Lie algebra $(L, [~,~], d_2)$;
\item[(iii)] the following compatibility condition holds
\begin{align*}
d_1 \theta_2 + d_2 \theta_1 + [\theta_1, \theta_2] = 0.
\end{align*}
\end{itemize}
\end{defi}

Like a differential graded Lie algebra can be twisted by a Maurer-Cartan element, the same result holds for bidifferential graded Lie algebras.

\begin{pro}\label{mc-deform}\cite{comp-lie} Let $(L, [~,~], d_1, d_2)$ be a bidifferential graded Lie algebra and let $(\theta_1, \theta_2)$ be a Maurer-Cartan element. Then the quadruple $(L, [~,~], d_1^{\theta_1}, d_2^{\theta_2})$ is a bidifferential graded Lie algebra, where
\begin{align*}
d_1^{\theta_1} = d_1 + [\theta_1, - ] ~~~ \text{ and } ~~~ d_2^{\theta_2} = d_2 + [\theta_2, - ].
\end{align*}
For any $\vartheta_1 , \vartheta_2 \in L^1$, the pair $(\theta_1 + \vartheta_1, \theta_2 + \vartheta_2)$ is a Maurer-Cartan element in the bidifferential graded Lie algebra $(L, [~,~], d_1, d_2)$ if and only if $(\vartheta_1, \vartheta_2)$ is a Maurer-Cartan element in the bidifferential graded Lie algebra $(L, [~,~], d_1^{\theta_1}, d_2^{\theta_2}).$
\end{pro}

%%%%%%%%%%%%%%%%%%%%%%%%%%%%%%%%%%%%%%%%%%%%%%%%%%%%%%%%%%%%%%%%%%%%%%%%%%%%
\section{Compatible Hom-associative algebras and their characterization}\label{section-ca}
%%%%%%%%%%%%%%%%%%%%%%%%%%%%%%%%%%%%%%%%%%%%%%%%%%%%%%%%%%%%%%%%%%%%%%%%%%%

In this section, we  introduce compatible Hom-associative algebras and then define compatible bimodules over them. We also construct a graded Lie algebra whose Maurer-Cartan elements are compatible Hom-associative structures.

%----------------------------------------------------------------------------
\subsection{Compatible Hom-associative algebras}
%----------------------------------------------------------------------------
In the following,  we  introduce compatible Hom-associative algebras and compatible bimodules over them.

\begin{defi}
A compatible Hom-associative algebra is a triple $(A, \mu_1, \mu_2,\a)$ in which $(A, \mu_1,\a)$ and $(A, \mu_2,\a)$ are both Hom-associative algebras satisfying the following compatibility
\begin{align} \label{comp-cond}
(a \cdot_1 b ) \cdot_2  \a(c) + ( a \cdot_2 b) \cdot_1 \a(c) = \a(a) \cdot_1 ( b \cdot_2 c ) + \a(a) \cdot_2 ( b \cdot_1 c ), ~ \text{ for } a, b, c \in A.
\end{align}
Here $\cdot_1$ and $\cdot_2$ are used for the multiplications $\mu_1$ and $\mu_2$, respectively.
\end{defi}

 In this case, we  say that $(\mu_1, \mu_2)$ is a compatible Hom-associative algebra structure on $A$ when the twisting map is clear from the context.

\begin{rmk}
It follows from the above definition that the sum $\mu_1 + \mu_2$ also defines a Hom-associative product on $A$. In other words, $(A, \mu_1 + \mu_2,\a)$ is a  Hom-associative algebra. In fact, one can show that $(A , \lambda_1 \mu_1 + \lambda_2 \mu_2,\a)$ is a Hom-associative algebra, for any $\lambda_1, \lambda_2 \in \mathbb{K}.$
\end{rmk}

\begin{defi}
Let $A = (A, \mu_1, \mu_2,\a)$ and $A' = (A', \mu_1', \mu_2',\a)$ be two compatible Hom-associative algebras. A morphism of compatible Hom-associative algebras from $A$ to $A'$ is a linear map $\phi : A \rightarrow A'$  which is a Hom-associative algebra morphism from $(A,\mu_1,\a)$ to $(A',\mu_1',\a)$ and a Hom-associative algebra morphism from $(A,\mu_2,\a)$ to $(A',\mu_2',\a)$.
\end{defi}
In what follow, we give some examples of compatible Hom-associative algebras.
\begin{ex}
Let  $(A, \mu_1, \mu_2)$ be compatible associative algebra and $\a: A \to A$ be a compatible associative algebra morphism, i.e. $\a$ is and associative algebra morphism for both $(A,\mu_1)$ and $(A,\mu_2)$. Then $(A, \a \circ \mu_1,\a \circ \mu_2 ,\a)$  is a compatible Hom-associative algebra.
\end{ex}

\begin{ex}
Let $(A,\mu_1,\mu_2,\a)$ be a compatible  Hom-associative algebra . Then for each $n \geq 0$, the  tuple $(A, \a^n \circ \mu_1, \a^n \circ \mu_2, \a^{n+1})$ is a compatible Hom-associative algebra. It is called the $n-th$ derived  compatible Hom-associative algebra of   $(A,\mu_1,\mu_2,\a)$.
\end{ex}

\begin{ex}
Let $(A, \mu,\a)$ be a Hom-associative algebra. A Nijenhuis operator on $A$ is a linear map $N : A \rightarrow A$ satisfying
\begin{align*}
& N \circ \a = \a \circ N,\\
& N(a) \cdot N(b) = N \big( N(a) \cdot b + a \cdot N(b) - N (a \cdot b) \big), \text{ for } a, b \in A.
\end{align*}
Nijenhuis operators are used to study linear deformations of Hom-associative algebras.
A Nijenhuis operator $N$ induces a new Hom-associative structure  on $(A,\a)$, denoted by $\mu_N : A \otimes A \rightarrow A, (a,b) \mapsto a \cdot_N b$ and it is defined by
\begin{align*}
a \cdot_N b := N(a) \cdot b + a \cdot N(b) - N (a \cdot b), ~ \text{ for } a, b \in A.
\end{align*}
Then it is easy to see that $(A, \mu, \mu_N,\a)$ is a compatible Hom-associative algebra.
\end{ex}

\begin{ex}
Let $(A, \mu,\a)$ be a  Hom-associative algebra and $(M,\ell,r,\b)$ be an $A$-bimodule. If $f \in C^2_{\a,\b} (A, M)$ is a Hochschild $2$-cocycle on $A$ with coefficients in the $A$-bimodule $M$, then $A \oplus M$ can be equipped with the $f$-twisted semidirect product Hom-associative algebra given by
\begin{align*}
(a, m) \cdot_f (b, n) = ( a \cdot b, a \lhd n + m \rhd b + f (a, b)), ~ \text{ for } (a, m), (b,n) \in A \oplus M.
\end{align*}
With this notation, it can be easily checked that  $(A \oplus M, \cdot_0 , \cdot_f,\a+\b)$ is a compatible Hom-associative algebra.
\end{ex}

The next class of examples come from compatible Rota-Baxter operators on Hom-associative algebras. See \cite{guo-book} for more on Rota-Baxter operators.

\begin{defi}
Let $(A, \mu,\a)$ be a Hom-associative algebra. A Rota-Baxter operator on $A$ is a linear map $R : A \rightarrow A$ satisfying $R \circ \a=\a \circ R$ and
\begin{align*}
R(a) \cdot R(b) = R \big( R(a) \cdot b + a \cdot R(b) \big), ~ \text{ for } a, b \in A.
\end{align*}
\end{defi}

A Rota-Baxter operator $R$ induces a new Hom-associative  product $\mu_R : A \otimes A \rightarrow A,~(a,b) \mapsto a \cdot_R b$ on the underlying vector space $A$ with the same twisting map $\a$  and it is given by
\begin{align*}
a \cdot_R b = R(a) \cdot b + a \cdot R(b), ~ \text{ for } a, b \in A.
\end{align*}

\begin{defi}
Two Rota-Baxter operators $R, S : A \rightarrow A$ on a Hom-associative algebra $(A,\mu,\a)$ are said to be compatible if for any $k,l \in \mathbb{K}$, the sum $k R + l S : A \rightarrow A$ is a Rota-Baxter operator on $(A,\mu,\a)$, or equivalently,
\begin{align*}
R(a) \cdot S (b) + S(a) \cdot R (b) =  R \big( S(a) \cdot b + a \cdot S(b) \big) + S \big( R(a) \cdot b + a \cdot R(b) \big), ~ \text{ for } a, b \in A.
\end{align*}
\end{defi}

It is straightforward to check the following result.

\begin{pro}
Let $R, S : A \rightarrow A$ be two compatible Rota-Baxter operators on  a Hom-associative algebra $(A,\mu,\a)$. Then $(A, \cdot_R,\cdot_S,\a )$  is a compatible Hom-associative algebra.
\end{pro}

\medskip

Let $A$ be a vector space and $\a: A \to A$ a linear map.  Consider the graded Lie algebra $(C^{\ast+1}_\a (A,A),[\cdot,\cdot]_\a)$.  Then   by Remark \ref{dgla-bdgla}, the quadruple  $(C^{\ast+1}_\a (A,A),[\cdot,\cdot]_\a,d_1=0,d_2=0)$ is a bidifferential graded Lie algebra. Then we have the following Maurer-Cartan characterization of compatible Hom-associative algebras.

\begin{pro}
There is a one-to-one correspondence between Maurer-Cartan elements in the bidifferential graded Lie algebra
$(C^{\ast+1}_\a (A,A),[\cdot,\cdot]_\a,0,0)$
and
compatible Hom-associative  algebra structures on $A$.
\end{pro}

\begin{proof}
Let $\mu_1$ and $\mu_2$ be two multiplicative  bilinear maps on $A$. In other words, $\mu_1,\mu_2 \in C_\a^2(A,A)$. then
\begin{align*}
\mu_1 \text{ is a Hom-associative structure on }\ A \Leftrightarrow~& ~~ [\mu_1, \mu_1]_\a = 0;\\
\mu_2 \text{ is a Hom-structure on }\ A  \Leftrightarrow~& ~~ [\mu_2, \mu_2]_\a = 0;\\
\text{ compatibility condition } (\ref{comp-cond}) ~\Leftrightarrow~& ~~  [\mu_1, \mu_2]_\a= 0.
\end{align*}
Hence, $(A, \mu_1, \mu_2, \alpha)$ is a compatible Hom-associative algebra  if and only if the pair $(\mu_1, \mu_2)$ is a Maurer-Cartan element in the bidifferential graded Lie algebra $(C^{\ast +1 }_\a (A,A), [~,~]_\a , d_1 = 0, d_2 = 0)$.

\end{proof}

\begin{pro}
Let $(A,\mu_1,\mu_2,\a)$ be a compatible Hom-associative algebra and $\mu'_1$ and $\mu'_2$ are two multiplicative bilinear maps on $A$. The tuple $(A, \mu_1+\mu'_1,\mu_2+\mu'_2,\a)$ is a compatible Hom-associative algebra if and only if the pair $(\mu'_1,\mu'_2)$ is a Maurer-Cartan element  in the bidifferential graded Lie algebra $(C^{\ast +1}_\a (A, A), [~,~]_\a, d_1 = [\mu_1, -], d_2 = [\mu_2, -])$.

\end{pro}

\begin{defi}
Let $A = (A, \mu_1, \mu_2,\a)$ be a compatible Hom-associative algebra. A compatible $A$-bimodule consists of  a tuple  $(M, \ell_1, r_1, \ell_2, r_2,\b)$ in which $M$ is a vector space and
\begin{align*}
\begin{cases}
\ell_1 : A \otimes M \rightarrow M,~ (a,m) \mapsto a \lhd_1 m, \\
r_1 : M \otimes A \rightarrow M,~ (m,a) \mapsto m \rhd_1 a,
\end{cases}
~~~~
\begin{cases}
\ell_2 : A \otimes M \rightarrow M,~ (a,m) \mapsto a \lhd_2 m, \\
r_2 : M \otimes A \rightarrow M,~ (m,a) \mapsto m \rhd_2 a,
\end{cases}
\end{align*}
are bilinear maps such that
\begin{enumerate}
\item  $(M, \ell_1, r_1,\b)$ is a bimodule over the Hom-associative algebra $(A, \mu_1,\a)$;

\item  $(M, \ell_2, r_2,\b)$ is a bimodule over the Hom-associative algebra $(A, \mu_2,\a)$;

\item  The following compatibilities are hold: for $a, b \in A$ and $m \in M$,
\begin{align}
( a \cdot_1 b ) \lhd_2 \b(m)  + ( a \cdot_2 b ) \lhd_1 \b(m)  = \a(a) \lhd_1 ( b \lhd_2 m) +\a( a) \lhd_2 ( b \lhd_1 m), \label{comp-bi1}\\
( a \lhd_1 m ) \rhd_2 \a(b)  + ( a \lhd_2 m ) \rhd_1 \a(b)  = \a(a) \lhd_1 ( m \rhd_2 b) + \a(a) \lhd_2 ( m \rhd_1 b), \label{comp-bi2}\\
( m \rhd_1 a ) \rhd_2 \a(b)  + ( m \rhd_2 a ) \rhd_1 \a(b)  = \b(m) \rhd_1 ( a \cdot_2 b) + \b(m) \rhd_2 ( a \cdot_1 b). \label{comp-bi3}
\end{align}
\end{enumerate}
\end{defi}

A compatible $A$-bimodule as above may be simply denoted by $M$ when no confusion arises.

\begin{ex}
Any compatible Hom-associative algebra $A = (A, \mu_1, \mu_2,\a)$ is a compatible $A$-bimodule in which $\ell_1 = r_1 = \mu_1$,  $\ell_2 = r_2 = \mu_2$ and $\a=\b$.
\end{ex}

\begin{rmk}
Let $A = (A, \mu_1, \mu_2,\a)$ be a compatible Hom-associative algebra and $(M, \ell_1, r_1, \ell_2, r_2,\b)$ be a compatible $A$-bimodule. Then it is easy to see that $(M, \ell_1 + \ell_2, r_1 + r_2,\b)$ is a bimodule over the Hom-associative algebra $(A, \mu_1 + \mu_2,\a).$
\end{rmk}

The following result generalizes the semidirect product for associative algebras \cite{loday-book}.

\begin{pro}\label{semi-prop}

Let $A = (A, \mu_1, \mu_2,\a)$ be a compatible Hom-associative algebra and $(M, \ell_1, r_1, \ell_2, r_2,\b)$ be a compatible $A$-bimodule.
Then the direct sum $A \oplus M$ carries a compatible Hom-associative algebra structure given by
\begin{align*}
(a,m) \bullet_1 (b, n) = &(a \cdot_1 b,~ a \lhd_1 n + m \rhd_1 b),\\
 (a,m) \bullet_2 (b, n) =& (a \cdot_2 b,~ a \lhd_2 n + m \rhd_2 b), \\
 (\a,\b)((a,m))=&(\a(a),\b(m))
\end{align*}
for $(a,m), (b,n) \in A \oplus M$. This is called the semidirect product.
\end{pro}

Given a  compatible Hom-associative algebra and a bimodule,  we can construct the dual bimodule as follows.
\begin{thm}
Let $(A,\mu_1,\mu_2,\a)$ be  a compatible Hom-associative algebra and $(M, \ell_1, r_1, \ell_2, r_2,\b)$ a bimodule.  Then $(M^*, r_1^\star-\ell_1^\star+r_2^\star-\ell_2^\star,-\ell_1^\star-\ell_2^\star,(\b^{-1})^\ast) $ is a bimodule of $A$  called the dual bimodule of  $(M, \ell_1, r_1, \ell_2, r_2,\b)$.

\end{thm}

%%%%%%%%%%%%%%%%%%%%%%%%%%%%%%%%%%%%%%%%%%%%%%%%%%%%%%%%%%%%%%%%%%%%
\section{Cohomology of compatible Hom-associative algebras}\label{sec:cohomology}
%%%%%%%%%%%%%%%%%%%%%%%%%%%%%%%%%%%%%%%%%%%%%%%%%%%%%%%%%%%%%%%%%%%%%

In this section, we introduce the cohomology of a compatible Hom-associative algebra with coefficients in a compatible bimodule. 

Let $(A, \mu_1, \mu_2,\a)$ be a compatible Hom-associative algebra and $ (M, \ell_1, r_1, \ell_2, r_2,\b)$ be a compatible $A$-bimodule. Let
\begin{align*}
{}^1\delta_{\a,\b} : C^n_{\a,\b}(A, M) \rightarrow C^{n+1}_{\a,\b} (A, M), ~ n \geq 1,
\end{align*}
 be the coboundary operator for the Hochschild cohomology of $(A, \mu_1,\a)$ with coefficients in the bimodule $(M, \ell_1, r_1,\b)$, and
\begin{align*}
{}^2\delta_{\a,\b} : C^n(_{\a,\b}A, M) \rightarrow C^{n+1}_{\a,\b} (A, M), ~ n \geq 1,
\end{align*}
 be the coboundary operator for the Hochschild cohomology of $(A, \mu_2,\a)$ with coefficients in the bimodule $(M, \ell_2, r_2,\b)$. Then, we obviously have
\begin{align*}
({}^1\delta_{\a,\b})^2 = 0 \qquad \text{ and } \qquad ({}^2\delta_{\a,\b})^2 = 0.
\end{align*}
Since the two Hom-associative structures $\mu_1$ and $\mu_2$ on $A$, and the corresponding bimodule structures on $M$ are compatible, it is natural that there is some compatibility
 between the coboundaries ${}^1\delta_{\a,\b}$ and ${}^2\delta_{\a,\b}$.
 Before we state the compatibility, we observe the followings.
 Let $(A,\mu_1,\mu_2,\a)$ be compatible Hom-associative algebra and $M$ be an $A$-bimodule.  We consider the  compatible Hom-associative algebra structure on the semidirect product
 $A \oplus M$ given in Proposition \ref{semi-prop}.
 We denote by
 $\pi_1, \pi_2 \in C^2_{\a+\b} (A \oplus M, A \oplus M)$ the elements corresponding to the Hom-associative products $\bullet_1$ and $\bullet_2$ on $A \oplus M$, respectively .
Let 
\begin{align*}
    \delta^1_{\a+\b}:  C^n_{\a+\b}(A \oplus M, A \oplus M) \to  C^{n+1}_{\a+\b}(A \oplus M, A \oplus M), n \geq 1, \\
     \delta^2_{\a+\b}:  C^n_{\a+\b}(A \oplus M, A \oplus M) \to  C^{n+1}_{\a+\b}(A \oplus M, A \oplus M), n \geq 1,
\end{align*}
denote respectively the coboundary operator for the  Hochschild cohomology of the Hom-associative algebra $(A\oplus M,\bullet_1,\a+\b)$ and $(A\oplus M,\bullet_2,\a+\b)$ with coefficients in themselves. 

Note that any map $f \in C^n(A, M)$ can be lifted  to a map $\widetilde{f} \in C^n (A \oplus M, A \oplus M)$ by
\begin{align*}
\widetilde{f} \big( (a_1, m_1), \ldots, (a_n, m_n)  \big) = \big( 0, f (a_1, \ldots, a_n ) \big),
\end{align*}
for $(a_i, m_i) \in A \oplus M$ and $i=1, \ldots, n$. Moreover, we have $f=0$ if and only if $\widetilde{f} = 0$. With all these notations, we have
\begin{align*}
\widetilde{({}^1\delta_{\a,\b} f )} =\delta^1_{\a+\b}(\widetilde{f})=  (-1)^{n-1}~[ \pi_1, \widetilde{f}]_\a  \qquad \text{ and } \qquad \widetilde{({}^2\delta f )} =\delta^2_{\a+\b}(\widetilde{f})=  (-1)^{n-1}~[ \pi_2, \widetilde{f}]_\a,
\end{align*}
for $f \in C^n_{\a,\b}(A, M)$. We are now ready to prove the compatibility condition satisfied by ${}^1\delta_{\a,\b}$ and ${}^2\delta_{\a,\b}$. More precisely, we have the following.

\begin{pro}\label{delta-comp}
The coboundary operators ${}^1\delta_{\a,\b}$ and ${}^1\delta_{\a,\b}$ are related by the following compatibility identity:
\begin{align*}
{}^1\delta_{\a,\b} \circ {}^2\delta_{\a,\b} + {}^2\delta_{\a,\b} \circ {}^1\delta_{\a,\b} = 0.
\end{align*}
\end{pro}

\begin{proof}
For any $f \in C^n_{\a,\b} (A,M)$, we have
\begin{align*}
&\widetilde{   ( {}^1\delta_{\a,\b} \circ {}^2\delta_{\a,\b} + {}^2\delta_{\a,\b} \circ {}^1\delta_{\a,\b})(f) } \\
&= \widetilde{{}^1\delta_{\a,\b} ( {}^2\delta_{\a,\b}{f})} ~ + ~ \widetilde{{}^2\delta_{\a,\b} ( {}^1\delta_{\a,\b} {f})} \\
&= (-1)^{n} ~  [\pi_1, \widetilde{{}^2\delta_{\a,\b}{f}}    ]_{\a} ~+~(-1)^{n} ~  [\pi_2, \widetilde{{}^1\delta_{\a,\b}{f}}    ]_{\a}     \\
&= - [\pi_1, [\pi_2, \widetilde{f}]_{\a} ]_{\a}  - [\pi_2, [\pi_1, \widetilde{f}]_\a ]_\a  \\
&= - [[\pi_1, \pi_2]_\a, \widetilde{f}]_\a ~+~ [\pi_2, [\pi_1, \widetilde{f}]_\a ]_\a - [\pi_2, [\pi_1, \widetilde{f}]_\a ]_\a \\
&= 0 ~~~~ \qquad (\because ~[\pi_1, \pi_2]_\a = 0).
\end{align*}
Therefore, it follows that  $  ( {}^1\delta_{\a,\b} \circ {}^2\delta_{\a,\b} + {}^2\delta_{\a,\b} \circ {}^1\delta_{\a,\b})(f) = 0$. Hence the result follows.
\end{proof}

We are now, in a position to define the cohomology of a compatible Hom-associative algebra $(A,\mu_1,\mu_2,\a)$ with coefficients in a bimodule $(M,\ell_1,r_1,\ell_2, r_2,\b)$. 

 We define the $n$-th cochain group $C^{n,c}_{\a,\b} (A, M)$, for each  $n \geq 1$, by
\begin{align*}
C^{n,c}_{\a,\b} (A, M) :=~ \underbrace{C^n_{\a,\b} (A, M) \oplus \cdots \oplus C^n_{\a,\b} (A, M)}_{n \text{ copies}}, ~ \text{ for } n \geq 1.
\end{align*}
Define a map $\delta^{n,c}_{\a,\b} : C^{n,c}_{\a,\b} (A, M) \rightarrow C^{n+1,c}_{\a,\b} (A, M)$, for $n \geq 1$, by
\begin{align}
\delta^{n,c}_{\a,\b} (f_1, \ldots, f_n ) :=~& ({}^1 \delta_{\a,\b} f_1, \ldots, \underbrace{{}^2\delta_{\a,\b} f_{i-1}+{}^1\delta_{\a,\b} f_i }_{i-\text{th place}}, \ldots, {}^2\delta_{\a,\b} f_n),\label{dc-2}
\end{align}
for $(f_1, \ldots, f_n ) \in C^{n,c}_{\a,\b} (A, M)$ and $2 \leq i \leq n$.

%

%The map $\delta^c_{\a,\b}$ can be understood by the following diagram:
%\begin{align*}
%\xymatrix{
% & & &  C^{3,c}_{\à,\b}(A,M) & \\
% & & C^{2,c}_{\a,\b}(A,M)  \ar[ru]^{{}^1\delta_{\a,\b}} \ar[rd]_{{}^2\delta_{\a,\b}} & & \\
%C^{0,c}_{\a,\b} (A,M) \ar[r]^{{}^1\delta_{\a,\b} = {}^2\delta_{\a,\b}} & C^{1,c}(A,M) \ar[ru]^{\delta_1} \ar[rd]_{\delta_2} & & C^3(A,M) & \cdots \cdot\\
% & & C^2(A,M)  \ar[ru]^{\delta_1} \ar[rd]_{\delta_2} & & \\
% & & & C^3(A,M) &  \\
%}
%\end{align*}
%Observe that, Proposition \ref{delta-comp} and the above diagrammatic presentation of $\delta_c$ shows that $(\delta_c)^2 =0$. However, we give a rigorous proof of the same fact.

\begin{pro}
The map $\delta^{n,c}_{\a,\b}$ is a coboundary operator, i.e., $\delta^{n+1,c}_{\a,\b} \circ \delta^{n,c}_{\a,\b}= 0$.
\end{pro}

\begin{proof}
For any $(f_1, \ldots, f_n) \in C^{n,c}_{\a,\b} (A,M)$, $n \geq 1$ and according to Proposition \ref{delta-comp},  we have
\begin{align*}
\delta^{n+1,c}_{\a,\b} \delta^{n,c}_{\a,\b}(f_1, \ldots, f_n)
&= \delta^{n+1,c}_{\a,\b} ({}^1 \delta_{\a,\b} f_1, \ldots, {}^2\delta_{\a,\b} f_{i-1}+{}^1\delta_{\a,\b} f_i , \ldots, {}^2\delta_{\a,\b} f_n) \\
&= \Big(    {}^1\delta_{\a,\b}  {}^1\delta_{\a,\b} f_1~,  {}^2\delta_{\a,\b}  {}^1\delta_{\a,\b} f_1 +  {}^1\delta_{\a,\b}  {}^2\delta_{\a,\b} f_1 +  {}^1\delta_{\a,\b}  {}^1\delta_{\a,\b} f_2~, \ldots, \\
&  \underbrace{   {}^2\delta_{\a,\b}  {}^2\delta_{\a,\b} f_{i-2} +  {}^2\delta_{\a,\b} {}^1\delta_{\a,\b} f_{i-1} + {}^1\delta_{\a,\b} {}^2\delta_{\a,\b} f_{i-1} + {}^1\delta_{\a,\b}  {}^1\delta_{\a,\b} f_i  }_{3 \leq i \leq n-1}~, \ldots, \\
&   {}^2\delta_{\a,\b} {}^2\delta_{\a,\b}  f_{n-1} + {}^2\delta_{\a,\b} {}^1\delta_{\a,\b} f_n +  {}^1\delta_{\a,\b} {}^2\delta_{\a,\b} f_n ~,~  {}^2\delta_{\a,\b}  {}^2\delta_{\a,\b} f_n \Big) \\
&= (0,0,\ldots,0)
\end{align*}
Then we have $\delta^{n+1,c}_{\a,\b} \circ \delta^{n,c}_{\a,\b}= 0$.
\end{proof}

Thus, we obtain  a cochain complex $\big( C^{\ast,c}_{\a,\b} (A, M), \delta^{\ast,c}_{\a,\b} \big)$. Let $Z^{n,c}_{\a,\b} (A, M)$ denote the space of $n$-cocycles and $B^{n,c}_{\a,\b} (A, M)$ the space of $n$-coboundaries. Then we have $B^{n,c}_{\a,\b} (A, M) \subset Z^{n,c}_{\a,\b} (A, M)$, for $n \geq 1$. The corresponding quotient groups
\begin{align*}
H^{n,c}_{\a,\b} (A, M) := \frac{ Z^{n,c}_{\a,\b} (A, M) }{ B^{n,c}_{\a,\b} (A, M)}, \text{ for } n \geq 1
\end{align*}
are called the cohomology groups of the compatible Hom-associative algebra $A$ with coefficients in the compatible $A$-bimodule $M$.

\begin{rmk}
When we consider compatible Hom-associative algebra as a bimodule over itself, we denote $C^{n,c}_{\a,\b} (A, M)$ as $C^{n,c}_{\a} (A, A)$ and $\delta^{n,c}_\alpha : C^{n,c}_{\a} (A, A) \to C^{n+1,c}_{\a} (A, A)$. We denote the $n$th cohomology group as $H^{n,c}_\a(A,A)$.
\end{rmk}

Let $(A, \mu_1, \mu_2,\a)$ be a compatible Hom-associative algebra and $(M,\ell_1,r_1,\ell_2,r_2,\b)$ be a compatible  $A$-bimodule. 
Then we know that $A^+=(A, \mu_1 + \mu_2,\a)$ is a Hom-associative algebra and $M^+=(M,\ell_1+\ell_2,r_1+r_2,\b)$ is an $A^+$-bimodule. 
Consider the cochain complex $\big(C^{\ast,c}_{\a,\b}(A,M),\delta^{\ast,c}_{\a,\b}\big)$  of the compatible Hom-associative algebra $A$ with coefficients in $M$ and the cochain complex $\big(C^{\ast}_{\a,\b}(A^+,M^+),\delta^{\ast}_{\a,\b}\big)$  of the  Hom-associative algebra $A^+$ with coefficients in $M^+$. 
For each $n \geq 1$, define the map 
\begin{align*}
  \Upsilon_n:   C^{n,c}_{\a,\b}(A,M) \to C^{n}_{\a,\b}(A^+,M^+),\ \mathrm{by}\ 
  \Upsilon_n(f_1,\ldots, f_n)=f_1+f_2+\ldots + f_n. 
\end{align*}
Observe that for any $(f_1,f_2,\ldots, f_n) \in  C^{n,c}_{\a,\b}(A,M)$, we have
\begin{align*}
    (\delta_{\a,\b} \circ \Upsilon_n)((f_1,\ldots,f_n))=&\delta_{\a,\b}(f_1+\ldots +f_n) \\
  =& ^1 \delta_{\a,\b}(f_1+\ldots + f_n)+ ^1 \delta_{\a,\b}(f_1+\ldots + f_n) \\
  =& \Upsilon_{n+1} ({}^1 \delta_{\a,\b} f_1, \ldots, \underbrace{{}^2\delta_{\a,\b} f_{i-1}+{}^1\delta_{\a,\b} f_i }_{i-\text{th place}}, \ldots, {}^2\delta_{\a,\b} f_n)\\
  =& \Upsilon_{n+1} \circ  \delta^{n,c}_{\a,\b}(f_1,\ldots,f_n).
\end{align*}
Therefore, we get the following
\begin{thm}
The collection $\{\Upsilon_n \}_{n \geq 1}$ defines a morphism of cochain complexes from $\big(C^{\ast,c}_{\a,\b}(A,M),\delta^{\ast,c}_{\a,\b}\big)$ to $\big(C^{\ast}_{\a,\b}(A^+,M^+),\delta^{\ast}_{\a,\b}\big)$. Hence, it induces a morphism $H^{\ast,c}_{\a,\b}(A,M) \to H^{\ast}_{\a,\b}(A^+,M^+)$
between corresponding cohomologies. 
\end{thm}

%--------------------------------------------------------------------
\subsection{Relation with the cohomology of compatible Hom-Lie algebras}
%---------------------------------------------------------------------

Recently, A. Das in \cite{Das comp HomLie} introduced a cohomology theory for compatible Hom-Lie algebras. In this paragraph, we show that the cohomology of compatible Hom-associative algebras is related to the cohomology of compatible Hom-Lie algebras. Let us first recall some results from the above-mentioned reference.
\begin{defi}
A compatible Hom-Lie algebra is a tuple $(\mathfrak{g}, [\cdot,\cdot]_1, [\cdot,\cdot]_2,\a)$ consists of a vector space $\mathfrak{g}$ and a linear map $\a: \mathfrak{g} \to \mathfrak{g}$  such that $(\mathfrak{g},[\cdot,\cdot]_1,\a)$ and $(\mathfrak{g},[\cdot,\cdot]_2,\a)$ are both Hom-Lie algebras and the following compatibility condition holds:
{\small\begin{align*}\label{comp-cond-Hom-Lie}
[[x,y]_1, \alpha (z)]_2 + [[y,z]_1, \alpha(x)]_2 + [[z,x]_1, \alpha (y)]_2 + [[x,y]_2, \alpha (z)]_1 + [[y,z]_2, \alpha(x)]_1 + [[z,x]_2, \alpha (y)]_1 = 0,
\end{align*}}
for all $x,y, z \in \mathfrak{g}.$
\end{defi}

\begin{defi}
Let $(\mathfrak{g}, [\cdot,\cdot]_1, [\cdot,\cdot]_2,\a)$ be a compatible Hom-Lie algebra. A compatible $\mathfrak{g}$-representation is a tuple $(V, \rho_1, \rho_2,\b)$, where $(V, \rho_1,\b)$ is a representation of the Hom-Lie algebra $(\mathfrak{g}, [\cdot,\cdot]_1,\a)$ and $(V, \rho_2,\b)$ is a representation of the Hom-Lie algebra $(\mathfrak{g}, [\cdot,\cdot]_2,\b)$ satisfying additionally
\begin{align*}
\rho_2 ([x, y]_1) \b(v)+ \rho_1 ([x, y]_2)\b(v) = \rho_1 (\a(x)) \rho_2 (y)v - \rho_1 (\a(y)) \rho_2 (x) v+ \rho_2 (\a(x)) \rho_1 (y)v - \rho_2 (\a(y)) \rho_1 (x),
\end{align*}
for  any $x, y \in \mathfrak{g}, v \in V$.
\end{defi}

Let  $(\mathfrak{g}, [\cdot,\cdot]_1, [\cdot,\cdot]_2,\a)$ be a compatible Hom-Lie algebra and $(V, \rho_1, \rho_2,\b)$ be a representation. There is a cochain complex $\{ C^{\ast,cL}_{\a,\b} (\mathfrak{g}, V),d^{\ast,cL}_{\a,\b})$ defined as follows:
\begin{align*}
C^{n,cL}_{\a,\b}  (\mathfrak{g}, V) :=~&  \underbrace{C^{n,L}_{\a,\b}  (\mathfrak{g}, V) \oplus \cdots \oplus  C^{n,L}_{\a,\b} (\mathfrak{g}, V)}_{n \text{ times}}, \text{ for } n \geq 1,
\end{align*}
where $C^{n,L}_{\a,\b}  (\mathfrak{g}, V) =\{f \in  \mathrm{Hom}(\wedge^n \mathfrak{g}, V)| f \circ \a^{\otimes n}=\b \circ f\}.$ The coboundary operator $d^{n,cL} : C^{n,cL}_{\a,\b}  (\mathfrak{g}, V) \rightarrow C^{n+1,cL}_{\a,\b}  (\mathfrak{g}, V)$ is given by
\begin{align*}
d^{n,cL}_{\a,\b} (f_1, \ldots, f_n ) := ( ^1 d^{n,L}_{\a,\b} f_1, \ldots, \underbrace{^1 d^{n,L}_{\a,\b} f_i + ^2 d^{n,L} f_{i-1}}_{i\text{-th place}}, \ldots, ^2 d^{n,L} f_n),
\end{align*}
for $(f_1, \ldots, f_n ) \in C^{n,cL}_{\a,\b} (\mathfrak{g}, V)$. Here $^1 d^{n,L}_{\a,\b}$ (resp. $^2 d^{n,L}_{\a,\b}$) is the coboundary operator for the Chevalley-Eilenberg cohomology of the Hom-Lie algebra $(\mathfrak{g}, [\cdot,\cdot]_1,\a)$ with coefficients in $(V, \rho_1,\b)$ ~~ (resp. of the Lie algebra $(\mathfrak{g}, [\cdot,\cdot]_2,\a)$ with coefficients in $(V, \rho_2,\b)$ ). The cohomology of the cochain complex $\{ C^{\ast,cL}_{\a,\b} (\mathfrak{g}, V), d^{\ast,cL}_{\a,\b} \}$ is called the cohomology  of the compatible Hom-Lie algebra $(\mathfrak{g}, [~,~]_1, [~,~]_2,\a)$ with coefficients in the compatible $\mathfrak{g}$-representation $(V, \rho_1, \rho_2,\b)$, and they are denoted by $H^{\ast,cL}_{\a,\b} (\mathfrak{g}, V).$

\medskip

On the other hand, it is a well-known fact that the standard skew-symmetrizatio gives rise to a map from the Hochschild cochain complex of a Hom-associative algebra to the Chevalley-Eilenberg cohomology complex of the corresponding sub-adjacent  Hom-Lie algebra (obtained by commutator). This can be generalized to compatible Hom-algebras as well.

Let $(A, \mu_1, \mu_2,\a)$ be a compatible Hom-associative algebra. Then it can be easily checked that the triple $(A, [\cdot,\cdot]_1, [\cdot,\cdot]_2)$
 is a compatible Hom-Lie algebra, where
 \begin{align*}
 [a, b]_1 := a \cdot_1 b - b \cdot_1 a ~~~~ \text{ and } ~~~~ [a, b]_2 := a \cdot_2 b - b \cdot_2 a, \text{ for } a, b \in A.
 \end{align*}
 We denote this compatible Hom-Lie algebra by $A_c$. Moreover, if $M$ is a compatible Hom-associative $A$-bimodule, then $M$ can be regarded as a compatible $A_c$-representation by
 \begin{align*}
 \rho_1 (a)(m) := a \lhd_1 m - m \rhd_1 a  ~~~~ \text{ and } ~~~~ \rho_2 (a)(m) := a \lhd_2 m - m \rhd_2 a, ~ \text{ for } a \in A_s, m \in M .
 \end{align*}
This compatible $A_c$-representation is denoted by $M_c$. With these notations, we have the following.
\begin{thm}
Let $(A,\mu_1,\mu_2,\a)$ be a compatible Hom-associative algebra and $M$ be a compatible $A$-bimodule. Then the standard skew-symmetrization
\begin{align*}
\Theta_n : C^{n,c}_{\a,\b} (A,M) \rightarrow C^{n,cL}_{\a,\b} (A_c, M_c),~ (f_1, \ldots, f_n) \mapsto (\overline{f_1}, \ldots, \overline{f_n}),~ \text{ for } n \geq 1, 
\end{align*}
where
\begin{align*}
\overline{f_i} ( a_1, \ldots, a_n ) = \sum_{\sigma \in \mathbb{S}_n} (-1)^\sigma~ f_i (a_{\sigma (1)}, \ldots, a_{\sigma(n)}),~~~ i = 1, \ldots, n,
\end{align*}
gives rise a morphism of cochain complexes. Hence it induces a map $\Theta^\ast : H^{\ast,c}_{\a,\b} (A, M) \rightarrow H^{\ast,cL}_{\a,\b} (A_c, M_c)$.
\end{thm}

\section{One-parameter formal deformation of compatible Hom-associative algebras}\label{section:deformation}

In this section, we study formal deformation theory of compatible Hom-associative algebras following the series of works by Gerstenhaber \cite{gers, gers2}.
\begin{defi}\label{deform defn}
Let $(A, \mu_1, \mu_2,\a)$ be a compatible Hom-associative algebra. A one-parameter formal deformation of $(A, \mu_1, \mu_2,\a)$ is a quadruple $(A[[t]], \mu_{1,t}, \mu_{2,t},\a_t)$, where
\begin{align*}
\mu_{1,t}, \mu_{2,t} : A[[t]]\times A[[t]] \to A[[t]]
\end{align*}
are $\mathbb{K}[[t]]$-bilinear maps, and 
\begin{align*}
\a_t : A[[t]] \to A[[t]]
\end{align*}
is a $\mathbb{K}[[t]]$-linear map such that
\begin{itemize}
\item[(i)] $\mu_{1,t} = \sum_{i\geq 0} \mu_{1,i} t^i,~ \mu_{2,t} = \sum_{i\geq 0} \mu_{2,i} t^i,~\text{and}~\a_t = \sum_{i\geq 0}\a_it^i$,where $\mu_{1,i}, \mu_{2,i} : A\times A \to A$ are $\mathbb{K}$-bilinear maps, $\a_i : A\to A$ is a $\mathbb{K}$-linear map for all $i\geq 0$.

\item[(ii)] $\mu_{1,0} = \mu_1, \mu_{2,0} = \mu_2$, and $\a_0=\a$ are the original multiplications and twisting maps respectively.

\item[(iii)] $(A[[t]], \mu_{1,t}, \a_t)$ and $(A[[t]], \mu_{2,t}, \a_t)$ are both Hom-associative algebras, that is, for all $a,b,c\in A$, we have
\begin{align}
&\mu_{1,t}(\mu_{1,t}(a, b),\a_t(c)) = \mu_{1,t}(\a_t(a), \mu_{1,t}(b, c)),\\
&\mu_{2,t}(\mu_{2,t}(a, b),\a_t(c)) = \mu_{2,t}(\a_t(a), \mu_{2,t}(b, c)).\label{deform ass-2}
\end{align}

\item[(iv)] $(A[[t]], \mu_{1,t}, \mu_{2,t},\a_t)$ satisfies the following compatibility conditions:
\begin{align}
 \mu_{1,t}(\mu_{2,t}(a,b), \a_t (c)) + \mu_{2,t}(\mu_{1,t}(a,b), \a_t (c))= \mu_{1,t}(\a_t(a),\mu_{2,t}(b,c)) + \mu_{2,t}(\a_t(a),\mu_{1,t}(b,c)), \label{deform com}
  \end{align}
 for all $a,b,c \in A$.
\end{itemize}
\end{defi}

The condition (4.1) and (\ref{deform ass-2}) is equivalent to the following equations. For all $n\geq 0$, we have
\begin{align}
&\sum_{i+j+k=n} \big(\mu_{1,i}(\mu_{1,j}(a, b),\a_k(c)) - \mu_{1,i}(\a_j(a), \mu_{1,k}(b, c))\big)=0 \label{deform ass-1-1}\\
&\sum_{i+j+k=n} \big(\mu_{2,i}(\mu_{2,j}(a, b),\a_k(c)) - \mu_{2,i}(\a_j(a), \mu_{2,k}(b, c))\big)=0.\label{deform ass-2-1}
\end{align}

Equivalently, we can write Equations \ref{deform ass-1-1} and \ref{deform ass-2-1} as follows:

\begin{align}
&\sum_{i+j+k=n} [\mu_{1,i}, \mu_{1,j}]_{\a_k} = 0,\label{deform ass-1-3}\\
&\sum_{i+j+k=n} [\mu_{2,i}, \mu_{2,j}]_{\a_k} = 0.\label{deform ass-2-3}
\end{align}

The condition \ref{deform com} is equivalent to the following equations. For all $n\geq 0$, we have
\begin{align}
&\sum_{i+j+k=n}\big( \mu_{2,i}(\mu_{1,j}(a,b), \a_k (c)) + \mu_{1,i}(\mu_{2,j}(a,b), \a_k (c))\big)\label{deform com 1} \\
\nonumber &= \sum_{i+j+k=n}\big( \mu_{1,i}(\a_j(a),\mu_{2,k}(b,c)) + \mu_{2,i}(\a_j(a),\mu_{1,k}(b,c))\big)
\end{align}

For all $n\geq 0$, we can re-write the Equation \ref{deform com 1} as follows:

\begin{align}
&\sum_{i+j+k=n} [\mu_{1,i}, \mu_{2,j}]_{\a_k} = 0. \label{deform com 2}
\end{align}

Therefore, using Equations \ref{deform ass-1-3}, \ref{deform ass-2-3}, and \ref{deform com 2}, we can say that $(A[[t]], \mu_{1,t}, \mu_{2,t},\a_t)$ is a one-parameter formal deformation of the compatible Hom-associative algebra $A$ if for all $n\geq 0$, and $a,b,c\in A$, it satisfies the following equations:
\begin{align*}
&\sum_{i+j+k=n} [\mu_{1,i}, \mu_{1,j}]_{\a_k} = 0,\\
&\sum_{i+j+k=n} [\mu_{2,i}, \mu_{2,j}]_{\a_k} = 0,\\
&\sum_{i+j+k=n} [\mu_{1,i}, \mu_{2,j}]_{\a_k} = 0.
\end{align*}

For $n=0$, we have 
$$[\mu_{1,0}, \mu_{1,0}]_{\a_0} = 0, ~ [\mu_{2,0}, \mu_{2,0}]_{\a_0} = 0,~[\mu_{1,0}, \mu_{2,0}]_{\a_0} = 0.$$
This is same as
$$[\mu_{1}, \mu_{1}]_{\a} = 0, ~ [\mu_{2}, \mu_{2}]_{\a} = 0,~[\mu_{1}, \mu_{2}]_{\a} = 0.$$
Note that the above relations are nothing but original Hom-associative and Hom-compatibilty relations.

For $n=1$, we have
\begin{align*}
& [\mu_{1,1}, \mu_{1,0}]_{\a_0} + [\mu_{1,0}, \mu_{1,1}]_{\a_0}+ [\mu_{1,0}, \mu_{1,0}]_{\a_1} = 0,\\
&[\mu_{2,1}, \mu_{2,0}]_{\a_0} + [\mu_{2,0}, \mu_{2,1}]_{\a_0} + [\mu_{2,0}, \mu_{2,0}]_{\a_1} = 0,\\
&[\mu_{1,1}, \mu_{2,0}]_{\a_0} + [\mu_{1,0}, \mu_{2,1}]_{\a_0} + [\mu_{1,0}, \mu_{2,0}]_{\a_1} = 0.
\end{align*}

Equivalently, we have 

\begin{align*}
& [\mu_{1,1}, \mu_{1}]_{\a} + [\mu_{1}, \mu_{1,1}]_{\a}+ [\mu_{1}, \mu_{1}]_{\a_1} = 0,\\
&[\mu_{2,1}, \mu_{2}]_{\a} + [\mu_{2}, \mu_{2,1}]_{\a} + [\mu_{2}, \mu_{2}]_{\a_1} = 0,\\
&[\mu_{1,1}, \mu_{2}]_{\a} + [\mu_{1}, \mu_{2,1}]_{\a} + [\mu_{1}, \mu_{2}]_{\a_1} = 0.
\end{align*}
As $[\mu_1,\mu_1]_{\alpha_1}=0$ and $[\mu_2,\mu_2]_{\alpha_2}=0$, we have
\begin{align*}
& [\mu_{1,1}, \mu_{1}]_{\a} = 0,\\
&[\mu_{2,1}, \mu_{2}]_{\a} = 0,\\
&[\mu_{1,1}, \mu_{2}]_{\a} + [\mu_{1}, \mu_{2,1}]_{\a}  = 0.
\end{align*}
Therefore, 
$$\delta^{2,c}_{\a}(\mu_{1,1}, \mu_{2,1})=0.$$
Thus, $(\mu_{1,1}, \mu_{2,1})$ is a $2$-cocyle in the cohomology of the compatible Hom-associative algebra $A$ with coeffients in itself. This pair $(\mu_{1,1}, \mu_{2,1})$ is called the infinitesimal of the deformation. This means the infinitesimal of the deformation is a $2$-cocycle. More generally, we have the following definition.

\begin{defi}
If $(\mu_{1,n}, \mu_{2,n})$ is the first non-zero term after $(\mu_{1,0}, \mu_{2,0})=(\mu_1, \mu_2)$ of the formal deformation $(\mu_{1,t}, \mu_{2,t})$, then we say that $(\mu_{1,n}, \mu_{2,n})$ is the $n$-infinitesimal of the deformation.
\end{defi}
Similar to the proof of showing that the infinitesimal is a $2$-cocycle, we have the following theorem.

\begin{thm}\label{infintesimal}
The $n$-infinitesimal is a $2$-cocycle.
\end{thm}

\subsection{Equivalent deformation and cohomology}
Suppose $A_t=(A,\mu_{1,t},\mu_{2,t}, \alpha_t)$ and $A'_t=(A,\mu'_{1,t},\mu'_{2,t}, \alpha_t)$ be two one-parameter compatible Hom-associative algebra deformations of $(A,\mu_1, \mu_2, \alpha)$, where $\mu_{1,t} = \sum_{i\geq 0} \mu_{1,i} t^i,~ \mu_{2,t} = \sum_{i\geq 0} \mu_{2,i} t^i,~\text{and}~\a_t = \sum_{i\geq 0}\a_it^i$ and $\mu'_{1,t} = \sum_{i\geq 0} \mu'_{1,i} t^i,~ \mu'_{2,t} = \sum_{i\geq 0} \mu'_{2,i} t^i$. 
\begin{defi}
 Two deformations $A_t$ and $A\rq_t$ are said to be equivalent if there exists a $\mathbb{K}[[t]]$-linear isomorphism $\Psi_t:A[[t]]\to A[[t]]$ of the form $\Psi_t=\sum_{i\geq 0}\psi_it^i$, where $\psi_0=id$ and $\psi_i:A\to A$ are $\mathbb{K}$-linear maps such that the following relations holds:
 \begin{align}\label{equivalent-1}
 &\Psi_t\circ \mu_{1,t}\rq=\mu_{1,t}\circ (\Psi_t\otimes \Psi_t),\\
 \label{equivalent-2}&\Psi_t\circ \mu_{2,t}\rq=\mu_{2,t}\circ (\Psi_t\otimes \Psi_t),\\
\label{equivalent-3}&\alpha_t\circ \Psi_t=\Psi_t\circ \alpha_t.
 \end{align}
 \end{defi}
 \begin{defi}
 A deformation $(\mu_{1,t}, \mu_{2,t},\alpha_t)$ of a compatible Hom-associative algebra $A$ is called trivial if $(\mu_{1,t}, \mu_{2,t},\alpha_t)$ is equivalent to the deformation $\mu_{1,0}, \mu_{2,0},\alpha_0)$, which is same as the undeformed one. A compatible Hom-associative algebra $A$ is called rigid if it has only trivial deformation upto equivalence.
 \end{defi}
 Equations \ref{equivalent-1}-\ref{equivalent-3} may be written as
 \label{equivalent 11}
\begin{align}
& \Psi_t(\mu\rq_{1,t}(a,b))=\mu_{1,t}(\Psi_t(a),\Psi_t(b)),\\
& \Psi_t(\mu\rq_{2,t}(a,b))=\mu_{2,t}(\Psi_t(a),\Psi_t(b)),\\
& \alpha_t(\Psi_t(a))=\Psi_t(\alpha_t(a)),\,\,\,\text{for all}~ a,b\in A.
 \end{align} 
 Note that the above equations are equivalent to the following equations:
 \begin{align}
 &\sum_{i\geq 0}\psi_i\bigg( \sum_{j\geq 0}\mu\rq_{1,j}(a,b)t^j \bigg)t^i=\sum_{i\geq 0}\mu_{1,i}\bigg( \sum_{j\geq 0}\psi_j(a)t^j,\sum_{k\geq 0}\psi_k(b)t^k \bigg)t^i,\\
  &\sum_{i\geq 0}\psi_i\bigg( \sum_{j\geq 0}\mu\rq_{2,j}(a,b)t^j \bigg)t^i=\sum_{i\geq 0}\mu_{2,i}\bigg( \sum_{j\geq 0}\psi_j(a)t^j,\sum_{k\geq 0}\psi_k(b)t^k \bigg)t^i,\\
 &\sum_{i\geq 0}\alpha_i\bigg( \sum_{j\geq 0}\psi_j(a)t^j \bigg)t^i=\sum_{i\geq 0}\psi_i\bigg( \sum_{j\geq 0}\alpha_j(a)t^j \bigg)t^i.
 \end{align}
 This is same as the following equations:
 \begin{align}
 \label{equivalent 10}&\sum_{i,j\geq 0}\psi_i(\mu\rq_{1,j}(a,b))t^{i+j}=\sum_{i,j,k\geq 0}\mu_{1,i}(\psi_j(a),\psi_k(b))t^{i+j+k},\\
 &\sum_{i,j\geq 0}\psi_i(\mu\rq_{2,j}(a,b))t^{i+j}=\sum_{i,j,k\geq 0}\mu_{2,i}(\psi_j(a),\psi_k(b))t^{i+j+k},\\
 &\label{equivalent 101}\sum_{i,j \geq 0}\alpha_i(\psi_j(a))t^{i+j}=\sum_{i,j\geq 0}\psi_i(\alpha_j(a))t^{i+j}.
 \end{align}
Using $\psi_0=Id$ and comparing constant terms on both sides of the above equations, we have
 \begin{align*}
 &\mu\rq_{1,0}(a,b)=\mu_1(a,b),\\
 &\mu\rq_{2,0}(a,b)=\mu_2(a,b),\\
 &\alpha_0(a)=\alpha(x).
  \end{align*}
  Now comparing coefficients of $t$, we have
  \begin{align}\label{equivalent main}
&\mu\rq_{1,1}(a,b)+\psi_1(\mu\rq_{1,0}(a,b))=\mu_{1,1}(a,b)+\mu_{1,0}(\psi_1(a),b)+\mu_{1,0}(a,\psi_1(b)),\\
&\mu\rq_{2,1}(a,b)+\psi_1(\mu\rq_{2,0}(a,b))=\mu_{2,1}(a,b)+\mu_{2,0}(\psi_1(a),b)+\mu_{2,0}(a,\psi_1(b)),\\
\label{equivalent main 1}&\alpha_1(a)+\alpha_0(\psi_1(a))=\alpha_1(a)+\psi_1(\alpha_0(a)).
  \end{align}
  The Equations (\ref{equivalent main})-(\ref{equivalent main 1}) are same as
  \begin{align*}
& \mu\rq_{1,1}(a,b)-\mu_{1,1}(a,b)=\mu_1(\psi_1(a),b)+\mu_1(a,\psi_1(b))-\psi_1(\mu_1(a,b))={}^1\delta_{ \alpha} \psi_1(a,b).\\
& \mu\rq_{2,1}(a,b)-\mu_{2,1}(a,b)=\mu_2(\psi_1(a),b)+\mu_2(a,\psi_1(b))-\psi_1(\mu_2(a,b))={}^2\delta_{ \alpha} \psi_1(a,b).\\
&\alpha(\psi_1(a)) - \psi_1(\alpha(a)) = 0.
\end{align*}
Thus, we have the following proposition.
  \begin{pro}
 Two equivalent deformations have cohomologous infinitesimals.
  \end{pro}
  \begin{proof}
  Suppose $A_t=(A,\mu_{1,t}, \mu_{2,t}, \alpha_t)$ and $A'_t=(A,\mu'_{1,t}, \mu'_{2,t}, \alpha_t)$ be two equivalent one-parameter formal deformations of compatible Hom-associative algebra $A$. Suppose $(\mu_{1,n}, \mu_{2,n})$ and $(\mu'_{1,n}, \mu'_{2,n})$ be two $n$-infinitesimals of the deformations $(\mu_{1,t}, \mu_{2,t}, \alpha_t)$ and $(\mu'_{1,t}, \mu'_{2,t}, \alpha_t)$ respectively. Using Equation (\ref{equivalent 10}) we get,
  \begin{align*}
  &\mu\rq_{1,n}(a,b)+\psi_n(\mu\rq_{1,0}(a,b))=\mu_{1,n}(a,b)+\mu_{1,0}(\psi_n(a),b)+\mu_{1,0}(a,\psi_n(b)),\\
  &\mu\rq_{1,n}(a,b)-\mu_{1,n}(a,b)=\mu_{1}(\psi_n(a),b)+\mu_1(a,\psi_n(b))-\psi_n(m\rq_1(a,b))={}^1\delta_{ \alpha} \psi_n(a,b),
  \end{align*}
and  
   \begin{align*}
  &\mu\rq_{2,n}(a,b)+\psi_n(\mu\rq_{2,0}(a,b))=\mu_{2,n}(a,b)+\mu_{2,0}(\psi_n(a),b)+\mu_{2,0}(a,\psi_n(b)),\\
  &\mu\rq_{2,n}(a,b)-\mu_{2,n}(a,b)=\mu_{2}(\psi_n(a),b)+\mu_2(a,\psi_n(b))-\psi_n(m\rq_2(a,b))={}^2\delta_{ \alpha} \psi_n(a,b). 
  \end{align*}
  Using equation (\ref{equivalent 101}) we get,
  \begin{align*}
&\alpha_0(\psi_n(a)) - \psi_0(\alpha_n(a)) + \alpha_n(\psi_0(a)) - \psi_n(\alpha_0(a))=0,\\
& \psi_n(\alpha(a)) - \alpha(\psi_n(a)) =0.
  \end{align*}
  Thus, infinitesimals of two deformations determines same cohomology class.
    \end{proof}
  \begin{thm}
  A non-trivial deformation of a compatible Hom-associative algebra is equivalent to a deformation whose infinitesimal is not a coboundary.
  \end{thm}
  \begin{proof}
 Let $(\mu_{1,t}, \mu_{2,t}, \a_t)$ be a deformation of the compatible Hom-associative algebra $A$ and $(\mu_{1,n}, \mu_{2,n})$ be the $n$-infinitesimal of the deformation for some $n\geq 1$. Then by Theorem (\ref{infintesimal}), $(\mu_{1,n}, \mu_{2,n})$ is a $2$-cocycle, that is, $\delta^{2,c}_\a (\mu_{1,n}, \mu_{2,n})=0$. Suppose $(\mu_{1,n}, \mu_{2,n})=-\delta^{1,c}_\a\phi_n$ for some $\phi_n\in C_\alpha^{1,c}(A, A)$, that is, $(\mu_{1,n}, \mu_{2,n})$ is a coboundary. We define a formal isomorphism $\Psi_t$ of $A[[t]]$ as follows:
  $$\Psi_t(a)=a+\phi_n(a)t^n.$$
  We set
  \begin{align*}
  &\bar{\mu_{1,t}}=\Psi^{-1}_t\circ \mu_{1,t}\circ (\Psi_t\otimes\Psi_t),\\
   &\bar{\mu_{2,t}}=\Psi^{-1}_t\circ \mu_{2,t}\circ (\Psi_t\otimes\Psi_t).
  \end{align*}
  Thus, we have a new deformation $(\bar{\mu_{1,t}}, \bar{\mu_{2,t}},\alpha_t)$ which is isomorphic to $(\mu_{1,t}, \mu_{2,t}, \a_t)$. By expanding the above equations and comparing coefficients of $t^n$, we get
  \begin{align*}
 & \bar{\mu_{1,n}}-\mu_{1,n}={}^1\delta_{ \alpha} \phi_n,\\
  & \bar{\mu_{2,n}}-\mu_{2,n}={}^2\delta_{ \alpha} \phi_n.
   \end{align*}
  Hence, $\bar{\mu_{1,n}}=0, ~\bar{\mu_{2,n}}=0$. By repeating this argument, we can kill off any infinitesimal which is a coboundary. Thus, the process must stop if the deformation is non-trivial. 
  \end{proof}
 \begin{cor}
 Let $(A, \mu_1, \mu_2,\a)$ be a compatible Hom-associative algebra.  If $H^{2,c}_\a (A, A)=0$ then $A$ is rigid.
 \end{cor}
\subsection{Obstruction and deformation cohomology}
Next, we define a deformation of order $n$ and discuss associated obstructions in extending the given deformation of a compatible Hom-associative algebra to a deformation of order $(n+1)$.
\begin{defi}
A deformation of order $n$ of a compatible Hom-associative algebra $A$ consist of a $\mathbb{K}[[t]]$-bilinear maps $\mu_{1,t}:A[[t]]\times A[[t]]\to A[[t]]$, $\mu_{2,t}:A[[t]]\times A[[t]]\to A[[t]]$, and a $\mathbb{K}[[t]]$-linear map $\alpha_t:A[[t]]\to A[[t]]$ of the forms
$$\mu_{1,t}=\sum^n_{i=0}\mu_{1,i}t^i,~~~ \mu_{1,t}=\sum^n_{i=0}\mu_{1,i}t^i,,~~~\alpha_t=\sum^n_{i=0}\alpha_it^i,$$
such that $(\mu_{1,t}, \mu_{2,t}, \a_t)$ satisfy all the conditions of a one-parameter formal deformation in the Definition \ref{deform defn} $(mod~ t^{n+1})$.
\end{defi}
A deformation of order $1$ is called an infinitesimal deformation. We say a deformation $(\mu_{1,t}, \mu_{2,t}, \a_t)$ of order $n$ of a compatible Hom-associative algebra is extendable to a deformation of order $(n+1)$ if there exists elements $\mu_{1, n+1}, \mu_{2, n+1}\in C_\alpha^{2,c}(A, A)$ and $\alpha_{n+1} \in C_\alpha^{1,c}(A, A)$ such that
\begin{align*}
&\bar{\mu_{1,t}}=\mu_{1,t}+\mu_{1, n+1}t^{n+1},\\
&\bar{\mu_{2,t}}=\mu_{2,t}+\mu_{2, n+1}t^{n+1},\\
&\bar{\alpha_t}=\alpha_t+\alpha_{n+1}t^{n+1},
\end{align*}
and $(\bar{\mu_{1,t}}, \bar{\mu_{2,t}}, \bar{\alpha_t})$ satisfies all the conditions of Definition \ref{deform defn} $(mod~ t^{n+2})$.

The deformation $(\bar{\mu_{1,t}}, \bar{\mu_{2,t}}, \bar{\alpha_t})$ of order $(n+1)$ gives us the following equations.
\begin{align}
&\sum_{i+j+k=n+1} \big(\mu_{1,i}(\mu_{1,j}(a, b),\a_k(c)) - \mu_{1,i}(\a_j(a), \mu_{1,k}(b, c))\big)=0 \label{obs ass-1-1}\\
&\sum_{i+j+k=n+1} \big(\mu_{2,i}(\mu_{2,j}(a, b),\a_k(c)) - \mu_{2,i}(\a_j(a), \mu_{2,k}(b, c))\big)=0.\label{obs ass-2-1}\\
&\sum_{i+j+k=n+1}\big( \mu_{2,i}(\mu_{1,j}(a,b), \a_k (c)) + \mu_{1,i}(\mu_{2,j}(a,b), \a_k (c))\big) \label{obs com}\\
\nonumber &= \sum_{i+j+k=n+1}\big( \mu_{1,i}(\a_j(a),\mu_{2,k}(b,c)) + \mu_{2,i}(\a_j(a),\mu_{1,k}(b,c))\big)
\end{align}
 This is same as the following equations
 
 \begin{align}
&\sum_{i+j+k=n+1} [\mu_{1,i}, \mu_{1,j}]_{\a_k} = 0,\\
&\sum_{i+j+k=n+1} [\mu_{2,i}, \mu_{2,j}]_{\a_k} = 0,\\
&\sum_{i+j+k=n+1} [\mu_{1,i}, \mu_{2,j}]_{\a_k} = 0.
\end{align}
Equivalently, we can rewrite the above equations as follows:
 \begin{align}
&{}^1\delta_{ \alpha}(\mu_{1,n+1}) =\frac{1}{2}\sum_{\substack{i+j+k=n+1\\i,j,k>0}} [\mu_{1,i}, \mu_{1,j}]_{\a_k},
\end{align}
\begin{align}
&{}^2\delta_{ \alpha}(\mu_{2,n+1})=\frac{1}{2}\sum_{\substack{i+j+k=n+1\\i,j,k>0}} [\mu_{2,i}, \mu_{2,j}]_{\a_k},
\end{align}
\begin{align}
& {}^2\delta_{ \alpha}(\mu_{1,n+1}) + {}^1\delta_{ \alpha}(\mu_{2,n+1})=\sum_{\substack{i+j+k=n+1\\i,j,k>0}} [\mu_{1,i}, \mu_{2,j}]_{\a_k}.
\end{align}
We define the $n$th obstruction to extend a deformation of a Hom-Leibniz algebra of order $n$ to a deformation of order $n+1$  as  $\text{Obs}^n = (\text{Obs}^n_{\mu_1}, \text{Obs}^n_{\mu_1, \mu_2}, \text{Obs}^n_{\mu_2})$, where
\begin{align}
\label{obs equ 222}&\text{Obs}^n_{\mu_1} =\frac{1}{2}\sum_{\substack{i+j+k=n+1\\i,j,k>0}} [\mu_{1,i}, \mu_{1,j}]_{\a_k},\\
&\text{Obs}^n_{\mu_2} :=\frac{1}{2}\sum_{\substack{i+j+k=n+1\\i,j,k>0}} [\mu_{2,i}, \mu_{2,j}]_{\a_k},\\
&\text{Obs}^n_{\mu_1, \mu_2} :=\sum_{\substack{i+j+k=n+1\\i,j,k>0}} [\mu_{1,i}, \mu_{2,j}]_{\a_k}.
\end{align}
Thus,  $(\text{Obs}^n_{\mu_1}, \text{Obs}^n_{\mu_1, \mu_2}, \text{Obs}^n_{\mu_2}) \in C^{3,c}_\a(A, A)$ and $(\text{Obs}^n_{\mu_1}, \text{Obs}^n_{\mu_1, \mu_2}, \text{Obs}^n_{\mu_2}) = \delta^{2,c}_\a(\mu_{1,n+1},\mu_{2,n+1})$.
\begin{thm}\label{obstruc-thm}
A deformation of order $n$ extends to a deformation of order $n+1$ if and only if the cohomology class of $\text{Obs}^n$ vanishes.
\end{thm}
\begin{proof}
Suppose a deformation $(\mu_{1,t},\mu_{2,t}, \alpha_t )$ of order $n$ extends to a deformation of order $n+1$. From the obstruction equations, we have
$$\text{Obs}^n = (\text{Obs}^n_{\mu_1},  \text{Obs}^n_{\mu_1, \mu_2}, \text{Obs}^n_{\mu_2}) = \delta^{2,c}_\a(\mu_{1,n+1},\mu_{2,n+1}).$$
As $\delta \circ \delta=0$, we get the cohomology class of  $\text{Obs}^n$ vanishes.

Conversely, suppose the cohomology class of  $\text{Obs}^n$ vanishes, that is,
$$\text{Obs}^n=\delta^{2,c}_\a (\mu_{1,n+1},\mu_{2,n+1}),$$
for some $2$-cochains $(\mu_{1, n+1}, \mu_{2,n+1})$. We define $(\mu'_{1,t},\mu'_{2,t} \alpha'_t)$ extending the deformation $(\mu_{1,t},\mu_{2,t}, \alpha_t)$ of order $n$ as follows-
\begin{align*}
&\mu'_{1,t}=\mu_{1,t}+\mu_{1,n+1}t^{n+1},\\
&\mu'_{2,t}=\mu_{2,t}+\mu_{2,n+1}t^{n+1},\\
&\alpha'_t=\alpha_t+ \alpha_{n+1}t^{n+1}.
\end{align*}
It is a routine work to check that  $(\mu'_{1,t}, \mu'_{2,t}, \alpha'_t)$ defines a formal deformation of order $n+1$.
 Thus, $(\mu'_{1,t}, \mu'_{2,t}, \alpha'_t)$ is a deformation of order $n+1$ which extends the deformation $(\mu_{1,t}, \mu_{2,t}, \alpha_t)$ of order $n$.
\end{proof}
\begin{cor}\label{obstruc-cor}
If $H^{3,c}_\a (A, A)=0$ then any infinitesimal deformation extends to a one-parameter formal deformation of $(A,\mu_1,\mu_2, \alpha)$.
\end{cor}

\section{Abelian extensions and cohomology}\label{section: abelian ext}
In this section, we show that the second cohomology group $H^{2,c}_{\alpha, \beta} (A, M)$ of a comatible Hom-associative algebra $(A,\mu_1,\mu_2, \alpha)$ with coefficients in a compatible bimodule $(M, l_1,r_1, l_2, r_2, \beta)$ can be interpreted as equivalence classes of abelian extensions of $A$ by $M$.

Let $A = (A, \mu_1,\mu_2, \alpha)$ be a compatible Hom-associative algebra and $M= (M, \beta)$ be a vector space equipped with a linear map $\beta : M \rightarrow M$. Note that $M$ can be considered as a compatible Hom-associative algebra with trivial multiplications.

\begin{defi}
An abelian extension of $A$ by $M$ is an exact sequence of compatible Hom-associative algebras
%\begin{align*}
\[
\xymatrix{
0 \ar[r] &  (M, 0, 0, \beta) \ar[r]^{i} & (E, \mu^E_{1}, \mu^E_{2},\alpha^E) \ar[r]^{j} & (A, \mu_1,\mu_2, \alpha) \ar[r] \ar@<+4pt>[l]^{s} & 0
}
\]
together with a $\mathbb{K}$-splitting (given by $s$) which satisfies 
\begin{align}\label{s-property}
\alpha^E \circ s = s \circ \alpha.
\end{align}
%\end{align*}
\end{defi}

An abelian extension induces a compatible $A$-bimodule structure on $(M, \beta)$ via the action map
\begin{align*}
\begin{cases}
&l_1(a, m) = \mu^E_1 (s(a), i(m))\\
&r_1(m,a)= \mu^E_1 ( i(m), s(a))
\end{cases}
;
\begin{cases}
&l_2(a, m) = \mu^E_2 (s(a), i(m))\\
&r_2(m,a)= \mu^E_2 ( i(m), s(a))
\end{cases}
\end{align*}

One can easily verify that this action in independent of the choice of $s$. 

%\begin{rmk}
%Let $(E, \alpha_E)$ and $(A, \alpha)$ be two vector spaces equipped with linear maps $\alpha_E : E \rightarrow E$ and $\alpha : A \rightarrow A$. Suppose $j : E \rightarrow A$ is a linear surjective map which commutes with respective structure maps. Then there might not be a section $s : A \rightarrow E$ of $j$ which commute with respective structure maps. Take $E = \langle x , y \rangle$ with $\alpha_E (x) = y,~ \alpha_E (y) = 0$ and $A = \langle a \rangle$ with $\alpha (a) = 0.$ Take $j (x) = a$ and $j (y) = 0$. Let $s$ be a section for $j$ commuting with respective structure maps. For $s (a) = \lambda x + \nu y$, we have $a = (j \circ s) (a) = \lambda a$, which implies that $\lambda = 1$. Finally,
%\begin{align*}
%0 = (s \circ \alpha) (a) = (\alpha_E \circ s)(a) = \alpha_E (x + \nu y) = y
%\end{align*}
%which is a contradiction.
%\end{rmk}

Two abelian extensions are said to be equivalent if there is a map $\phi : E \rightarrow E'$ between compatible Hom-associative algebras making the following diagram commute
\[
\xymatrix{
0 \ar[r] &  (M, 0,0, \beta) \ar[r]^{i} \ar@{=}[d] & (E, \mu^E_1,\mu^E_2, \alpha^E) \ar[d]^{\phi} \ar[r]^{j} & (A, \mu_1,\mu_2, \alpha) \ar[r] \ar@{=}[d] \ar@<+4pt>[l]^{s} & 0 \\
0 \ar[r] &  (M, 0, 0,\beta) \ar[r]^{i'} & (E', \mu'^{E}_1, \mu'^{E}_2, \alpha'^E) \ar[r]^{j'} & (A, \mu_1,\mu_2, \alpha) \ar[r] \ar@<+4pt>[l]^{s'} & 0 .
}
\]
Observe that two extensions with same $i$ and $j$ but different $s$ are always equivalent.

Suppose $M$ is a given $A$-bimodule. We denote by $\mathcal{E}xt _c(A, M)$ the equivalence classes of abelian extensions of $A$ by $M$ for which the induced $A$-bimodule structure on $M$ is the pescribed one.

The next result is inspired from the classical case.
\begin{thm}\label{thm-abelian-ext}
$H^{2,c}_{\alpha, \beta} (A, M) \cong \mathcal{E}xt_c (A, M).$
\end{thm}

\begin{proof}
Given a $2$-cocycle $f \in C^{2,c}_{\alpha, \beta} (A, M)$, we consider the $\mathbb{K}$-module $E = M \oplus A$ with following structure maps
\begin{align*}
{\mu}^E_1 ((m, a), (n, b)) =~& ( r_1(m , b) + l_1( a , n) + f (a, b),~ \mu_1 (a, b)),\\
{\mu}^E_2 ((m, a), (n, b)) =~& ( r_2(m , b) + l_2( a , n) + f (a, b),~ \mu_2 (a, b)),\\
{\alpha}^E((m, a)) =~& (\beta (a), \alpha (a)).
\end{align*}
(Observe that when $f =0$ this is the semi-direct product.)
Using the fact that $f$ is a $2$-cocycle, it is easy to verify that $(E, \mu^E_1, \mu^E_2, \alpha^E)$ is a compatible Hom-associative algebra. Moreover, $0 \rightarrow M \rightarrow E \rightarrow A \rightarrow 0$ defines an abelian extension with the obvious splitting. Let $(E' = M \oplus A, \mu'^{E}_1, \mu'^{E}_2, \alpha'^E)$ be the corresponding compatible Hom-associative algebra associated to the cohomologous $2$-cocycle $f - \delta^{1,c}_{\alpha, \beta} (g)$, for some $g \in C^{1,c}_{\alpha, \beta} (A, M)$. The equivalence between abelian extensions $E$ and $E'$ is given by $E \rightarrow E'$, $(m, a) \mapsto (m + g (a), a)$. Therefore, the map $H^2_{\alpha, \beta} (A, M) \rightarrow \mathcal{E}xt_c (A, M) $ is well defined.

Conversely, given an extension 
$0 \rightarrow M \xrightarrow{i} E \xrightarrow{j} A \rightarrow 0$ with splitting $s$, we may consider $E = M \oplus A$ and $s$ is the map $s (a) = (0, a).$ With respect to the above splitting, the maps $i$ and $j$ are the obvious ones. Moreover, the property (\ref{s-property}) then implies that $\alpha^E = (\beta, \alpha)$. Since $j \circ \mu^E_1 ((0, a), (0, b)) = \mu_1 (a, b)$, and $j \circ \mu^E_2 ((0, a), (0, b)) = \mu_2 (a, b)$ as $j$ is an algebra map, we have $\mu^E_1 ((0, a), (0, b)) = (f (a, b), \mu_1 (a, b))$, and $\mu^E_2 ((0, a), (0, b)) = (f (a, b), \mu_2 (a, b))$,  for some $f \in C^{2,c}_{\alpha, \beta} (A, M).$ The Hom-associativity of $\mu^E_1, \mu^E_2$ then implies that $f$ is a $2$-cocycle. Similarly, one can observe that any two equivalent extensions are related by a map $E = M \oplus A \xrightarrow{\phi} M \oplus A = E'$, $(m, a) \mapsto (m + g(a), a)$ for some $g \in C^{1,c}_{\alpha, \beta} (A, M)$. Since $\phi$ is an algebra morphism, we have
\begin{align*}
&\phi \circ \mu^E_1 ((0, a), (0, b)) = \mu'^{E}_1 (\phi (0, a) , \phi (0, b))\\
&\phi \circ \mu^E_2 ((0, a), (0, b)) = \mu'^{E}_2 (\phi (0, a) , \phi (0, b))
\end{align*}
which implies that $f' (a, b) = f (a, b) - (\delta_{\alpha, \beta} ~g)(a, b)$. Here $f'$ is the $2$-cocycle induced from the extension $E'$. This shows that the map $\mathcal{E}xt_c (A, M) \rightarrow H^{2,c}_{\alpha, \beta} (A, M)$ is well defined. Moreover, these two maps are inverses to each other.
%$\xymatrix{
%0 \ar[r] & M \ar[r]^{i} & E \ar[r]^{j} & A \ar[r] & 0
%}$
\end{proof}

%\section{Remark and future work}

%%%%%%%%%%%%%%%%%%%%%

\end{document}